\documentclass[oneside,english]{amsart}
\usepackage[T1]{fontenc}
\usepackage[latin9]{inputenc}
\usepackage{geometry,verbatim}
\usepackage{mathrsfs}
\usepackage{amstext}
\usepackage{amsthm}
\usepackage{amssymb,amsrefs}
\usepackage{hyperref}

\makeatletter

\numberwithin{equation}{section}
\numberwithin{figure}{section}
\theoremstyle{definition}
\newtheorem*{defn*}{\protect\definitionname}
\theoremstyle{plain}
\newtheorem{thm}{\protect\theoremname}[section]
\theoremstyle{definition}
\newtheorem{defn}[thm]{\protect\definitionname}
\theoremstyle{plain}
\newtheorem{cor}[thm]{\protect\corollaryname}
\theoremstyle{plain}
\newtheorem{lem}[thm]{\protect\lemmaname}
\makeatother

\usepackage{babel,color}
\DeclareMathOperator{\reg}{reg}

\providecommand{\corollaryname}{Corollary}
\providecommand{\definitionname}{Definition}
\providecommand{\lemmaname}{Lemma}
\providecommand{\theoremname}{Theorem}

\begin{document}
\title{Bounds on regularity of quadratic monomial ideals}
\date{\today}
\author{Grigoriy Blekherman}
\address{School of Mathematics, Georgia Institute of Technology, Atlanta, GA 30332}
\email{greg@math.gatech.edu}
\author{Jaewoo Jung}
\address{School of Mathematics, Georgia Institute of Technology, Atlanta, GA 30332}
\email{jjung325@math.gatech.edu}

\begin{abstract} 
Castelnuovo-Mumford regularity is a measure of algebraic complexity of an ideal. Regularity of monomial ideals can be investigated combinatorially. We use a simple graph decomposition and results from structural graph theory to prove, improve and generalize many of the known bounds on regularity of quadratic square-free monomial ideals.
\end{abstract}
\maketitle
\section{Introduction}

We consider bounds on Castelnuovo-Mumford regularity of a square-free quadratic monomial ideal $I$ over a field of characteristic $0$.
Many recent papers investigated regularity of such ideals 
\cite{MR3070118}\cite{fakhari2017regularity}\cite{MR3213523}\cite{MR2790928}\cite{MR2943752}\cite{MR2563591}, see also \cite{MR2932582} for a survey.
One can associate to a quadratic square-free monomial ideal $I$ a graph $G$, whose vertices are the variables, and edges correspond to quadratic generators of $I$. Therefore, quadratic square-free monomial ideals are often called edge ideals in the literature.

Another popular approach, which we follow, is to associate to $I$ a graph $G(=G(I))$ where quadratic generators of $I$ are the non-edges of $G$. We note that the ideal $I(G)$ is the Stanley-Reisner ideal of the clique complex of $G$ \cite{MR1453579}*{Chapter 2}. We use $\reg (G)$ to denote Castelnuovo-Mumford regularity of the non-edge ideal $I(G)$ of $G$. 

Our main tool for bounding regularity is the following decomposition theorem, which is based on a straightforward application of Hochster's formula \cite{MR0441987}.
\begin{thm}\label{1.mainthm}
	Let $G$ be a graph. Let $G_{1}$ and $G_{2}$ be subgraphs which cover cliques of $G$ (i.e. any clique of $G$ is a clique in either $G_{1}$ or $G_{2}$.) Then, \[\reg(G)\leq\max\{\reg(G_{1}),\reg(G_{2}),\reg(G_{1}\cap G_{2})+1\}\]
\end{thm}

We first apply this theorem to the case of a separator of $G$, i.e. a subset of vertices of $G$ whose deletion disconnects $G$. For a subgraph $H$ of $G$ we use $G\setminus H$ to denote the induced subgraph on vertices of $G$ that are not in $H$.

\begin{thm}[Cutset/Separator Decomposition]\label{1.Separator}
 Let $T$ be an induced subgraph of $G$ such that $G\setminus T$ is disconnected. Let $C_{1},\dots,C_{k}$ be the connected components of $G\setminus T$. Then, \[\reg(G)\leq\max\{\reg(G_{i})_{i=1,\dots,k},\reg(T)+1\}\] where $G_i$ are the induced subgraphs on vertices of $C_i$ and $T$, for $i=1,\dots,k$.
\end{thm}

Theorem \ref{1.Separator} generalizes a decomposition result used by Dao, Huneke and Schweig in \cite{MR3070118}*{Lemma 3.1}. Recall that an open neighborhood $N_G(v)$ of a vertex $v$ is the induced subgraph on the vertices adjacent to $v$, and a closed neighborhood $N_G[v]$ of $v$ is the induced subgraph on $v$ and all vertices adjacent to $v$. Decomposition in \cite{MR3070118} arises as a special, but very useful case, where $T$ is the open neighborhood of a vertex $v$. An additional simplification comes from the fact that regularity of the open and closed neighborhoods of $v$ are the same.

\begin{thm}[Vertex Decomposition]\label{1.vertex}
Let $v$ be a vertex of $G$. Then, \[\reg(G)\leq\max\{\reg(G \setminus v),\reg(N_{G}(v))+1\}.\] 
\end{thm}

The above decompositions allow us to leverage existing results in structural graph theory results to derive a number of interesting consequences. 
A family of graphs is called \textit{hereditary} if it is closed under vertex deletion \cite{MR3408128}*{Chapter 2}. 
 Recall that a graph is chordal if it does not contain a cycle of length at least four as an induced subgraph \cite{MR2159259}*{Section 5.5}. Chordal graphs form a hereditary family, and moreover in every chordal graph there exists a vertex whose neighborhood is the complete graph \cite{MR2159259}*{Proposition 5.5.1}. 
Therefore, we immediately obtain a result of Fr\"{o}berg that regularity of chordal graphs is at most $2$. With the same idea we obtain the following theorem:

\begin{thm}[Hereditary theorem] \label{thm:hereditary} Let $\mathscr{G}$ be a hereditary family with the following property:
there exists $t\in \mathbb{N}$, such that for any $G \in \mathscr{G}$ there is a separator $G'$ of $G$ with $\reg (G') \leq t$. Then regularity of any $G \in \mathscr{G}$ is at most $t+1$.
\end{thm}
An induced chordless cycle of length at least four is called a \textit{hole}. 
An interesting connection between notions in algebra and structural graph theory was found by Eisenbud, Green, Hulek, and Popescu. A projective subscheme $X\subseteq \mathbb{P}^r$ satisfies Green-Lazarsfeld condition $N_{2,p}$  for integer $p\geq 1$ if the ideal $I(X)$ of $X$ is generated by quadratics and the first $(p-1)$-steps 
of the minimal free resolution of the ideal $I(X)$ are linear. 
It was shown in \cite{MR2188445}*{Theorem 2.1} that a non-edge ideal $I$ satisfies condition $N_{2,p}$ for some integer $p\geq 2$ if and only if $G$ does not contain a hole of length at most $p+2$. 
 Additionally, the resolution is often simpler if the ideal satisfies property $N_{2,p}$ with $p \geq 2$,  (see discussion in \cite{MR3070118}). This motivates us to look at graphs with restrictions on holes. We say that a graph satisfies condition $N_{2,p}$ if the corresponding non-edge ideal $I$ does. 

A highly studied family of graphs are \textit{perfect} graphs. A graph is perfect if its chromatic number is equal to its clique number, and the same is true for every induced subgraph \cite{MR2159259}*{Section 5.5}. Thus, perfect graphs are simple from the point of view of coloring. One of the most celebrated results in structural graph theory is the Strong Perfect Graph Theorem \cite{MR2233847}, which states that $G$ is perfect if and only if $G$ and its complement do not contain odd holes. Using Corollary \ref{thm:hereditary} we show that perfect graphs satisfying property $N_{2,2}$ have regularity at most $3$, since they form a hereditary family, and it is known that $4$-hole free perfect graphs have a vertex with a chordal neighborhood in \cite{MR1852504}. We note that without property $N_{2,2}$ a perfect graph on $2n$ vertices can have regularity $n+1$ (See Section \ref{sec:genus} for details).

\begin{cor}
If a perfect graph $G$ does not contain a hole of length four, then regularity of $G$ is at most $3$.
\end{cor}

We observe that graphs without even holes form a hereditary family, and it was shown in \cite{MR2292535} that an even-hole free graph contains a vertex with a chordal neighborhood. Therefore we obtain the following corollary:
\begin{cor}\label{1.evenhole}
If the graph $G$ is even-hole free, then regularity of $G$ is at most $3$.
\end{cor}

We also generalize a result of Nevo \cite{MR2739498}*{Theorem 5.1}. 
Let $F$ be the graph on $5$ vertices, consisting of an isolated vertex and two triangles sharing one edge. If $G$ satisfies condition $N_{2,2}$ and does not contain $F$ as an induced subgraph then regularity of $G$ is at most $3$.
\begin{cor}\label{1.fan}
Let $G$ be a graph satisfying $N_{2,2}$ which does not contain $F$ as an induced subgraph. Then regularity of $G$ is at most $3$.
\end{cor}
Nevo's result assumes that $G$ satisfies $N_{2,2}$ and does not contain the union of an isolated vertex with a triangle as an induced subgraph, which is a stronger condition on $G$. 

So far we only considered decompositions which use induced subgraphs of $G$, but now we consider an interesting decomposition where subgraphs $G_1$ and $G_2$ are not induced. Let $M$ be a subgraph of $G$ (not necessarily induced). We use $G-M$ to denote the subgraph of $G$ obtained by deleting edges of $M$. Let $G_{M}$ be the induced subgraph on vertices of $M$ and vertices of $G$ which are adjacent to both vertices of an edge in $M$. In other words, $G_{M}=G[V(M)\cup W]$ where $W$ is the subset of vertices of $G$ given by $k\in W$ if $ik,jk\in E(G)$ for some $ij\in E(M)$.

\begin{thm} \label{1.subgraphdecomp}
Let $G$ be a graph and $M$ be a subgraph in $G$.  Then, we have 
\[\reg(G)\leq\max\{\reg(G-M), \reg (G_{M}),\reg(G_{M}-M)+1\}.\] 
\end{thm}

As a special case, by taking $M$ to be an edge $e=ij$ in $G$ in Theorem \ref{1.subgraphdecomp},
we get the following edge-neighborhood decomposition theorem.

\begin{thm}[Edge-neighborhood decomposition] \label{edgenbhd}
 Let $G$ be a graph and $e=ij$ be an edge in $G$ with vertices $i$ and $j$. 
 Then,  
 \[\reg(G)\leq\max\{\reg(G-e),\reg(G_{e}-e)+1\}.\] 
\end{thm}
In \cite{MR3199032}, Fern\'{a}ndez-Ramos and Gimenez classified bipartite graphs whose edge ideals have regularity at most $3$. Using Corollary \ref{edgenbhd} and structural graph theory results we quickly recover their classification in Theorem \ref{thm:bipartite} 
Another measure of complexity of a graph is its \textit{genus}, which is defined as the genus of the smallest orientable surface on which $G$ can be drawn in such a way that edges of $G$ intersect only at the vertices \cite{MR1852593}. The famous case of planar graphs corresponds to genus $0$. Woodroofe showed in \cite{MR3249840} that planar graphs have regularity at most $4$ and this bound is tight. We generalize this bound to arbitrary genus:

\begin{thm}
Let $S_{g}$ be the orientable 2-dimensional manifold of genus $g$. Suppose that a graph $G$ is embedded into $S_{g}$. Then, regularity of $G$ is at most $\lfloor1+\sqrt{1+3g}\rfloor+2$.
\end{thm}
Note that this bound is also tight. 
Indeed, it is known by \cite{MR0349461} that the genus of the complete $m$-partite graph with every part of size two, $K_{m(2)}(=K_{2,2,\dots,2})$ is $\frac{(m-3)(m-1)}{3}$ for $m\not\equiv 2 \, (\text{mod} 3)$. Then, $\reg (K_{m(2)})\leq\lfloor1+\sqrt{1+3g}\rfloor+2=m+1$ which is tight.

Let $n$ be the number of vertices of $G$. It is well known (see \cite{MR1417301}*{Lemma 2.1}) that $\reg(G) \leq \frac{n}{2} +1$, and this is tight, by considering $G=K_{\frac{n}{2}(2)}$ with $n$ even. However, as shown in \cite{MR3070118}, the bounds on regularity become much better if $G$ satisfies condition $N_{2,p}$ for some $p \geq 2$. Specifically, if $G$ satisfies $N_{2,p}$ for $p \geq 2$, then $$\reg(G)\leq\log_{\frac{p+3}{2}}\frac{n-1}{p}+3.$$

We prove the following upper bound of regularity of $G$, which slightly improves on the bound above.
\begin{thm}
	If $G$ satisfies property $N_{2,p}$ for $p \geq 2$, then,
	\[ 
    \reg(G)\leq \log_{\frac{p+4}{2}} \frac{n}{p+1} +4.	
    \]
\end{thm}

\section{Graphs and Clique Complexes}\label{sec:2}

A simple graph $G$ consists of the vertex set $V(G)$ and the edge set $E(G)$. For a subgraph $G'$ of $G$ we use $G\setminus G'$ to denote the induced subgraph on $V(G)\setminus V(G')$, and $G-G'$
to denote the subgraph of $G$ obtained by deleting all edges of $G'$ (but no vertices). An induced subgraph on a subset of vertices $W$ is denoted by $G[W]$. 

A simple graph $G$ with vertex set $[n]$ can be identified with a square-free quadratic monomial ideal $I(G)$ over a field $k$ via $I(G) = \langle x_i x_j \, |\, ij \notin E \rangle $. We call $I(G)$ the non-edge ideal of $G$.

For a vertex $v$ of $G$, we define the open neighborhood of $v$, denoted by $N_G(v)$, to be the induced subgraph of $G$ on vertices which are adjacent to $v$. We also define the closed neighborhood of $v$, denoted by $N_G[v]$, to be the induced subgraph  on $v$ and the vertices adjacent to $v$. 
Now we introduce the clique complex of a graph. 
\begin{defn*}
Given a graph $G$, the clique complex of $G$, denoted by $\Delta G$, is a simplicial complex that consists of $t$-simplices $(x_{i_{1}},x_{i_{2}},\dots,x_{i_{t}})$ whenever $G[\{x_{i_{1}},x_{i_{2}},\dots,x_{i_{t}}\}]=K_{t}$ where $K_t$ is the complete graph on $t$ vertices.
\end{defn*}
 We observe that the non-edge ideal $I(G)$ is the Stanley-Reisner ideal of the clique complex $\Delta G$ \cite{MR1453579}*{Chapter 2}. The graph $G$ associated with the clique complex $\Delta G$ is uniquely determined as the closure of the $1$-skeleton of $\Delta G$. Thus, there is a one-to-one correspondence between the family of simple graphs $G$ and the family of clique complexes $\Delta G$ given by $G\mapsto\Delta G$ and $F\mapsto G(F)$. 
Homology of clique complexes $\Delta G$ gives Betti numbers of non-edge ideal $I(G)$ (or equivalently the Stanley-Reisner rings $k[\Delta]:=k[x_1,\dots,x_n]/I(G)$) via Hochster's formula, if the characteristic of the field $k$ is $0$ (see \cite{MR3213521}*{Remark 7.15}).

\begin{thm}[Hochster]\label{Hochster}
 Let $I(G)$ be the non-edge ideal of a graph $G$.
Then for $t\geq i+2$, 
\[
\beta_{i,t}(I(G))=\sum_{|W|=t}\dim_{k}(\tilde{H}_{t-i-2}(\Delta G[W])),
\]
where $W$ runs over all subsets of the vertex set of $G$ of size $t$.
\end{thm}

We refer to \cite{MR1453579}*{Corollary 4.9} and \cite{MR3213521}*{Theorem 7.11} for Hochster's formula.  We define regularity of a graph $G$ to be the Castelnuovo-Mumford regularity of the corresponding ideal $I(G)$ \cite{MR2560561}*{Definition 18.1}.
\begin{defn}
Let $\reg(G)$ be the Castelnuovo-Mumford regularity of the non-edge ideal $I(G)$ of graph $G$.
In detail, $\reg(G):=\max\{r|\beta_{i,i+r}(I(G))\not= 0 \text{ for some }i\}$. Regularity of the complete graph, which corresponds to the empty ideal, is one.
\end{defn}
By Hochster's formula we see that regularity of a graph $G$ is the smallest integer $q\geq 1$ such that $(q-2)$-nd (reduced) homology of the clique complex of any induced subgraph of $G$ is non-zero. Note that regularity of the Stanley-Reisner ring $k[\Delta]$ is one less than regularity of the non-edge ideal $I(G)$.

\section{Graph Decompositions} \label{sec:mainthm}
Since Betti numbers of a Stanley-Reisner ideal can be obtained by calculating homology of subcomplexes of the corresponding clique complex, two subgraphs that cover all cliques of $G$ give us enough information to bound regularity of $G$.

\begin{thm}\label{thm:maindecomp}
Let $G$ be a graph. Let $G_{1}$ and $G_{2}$ be subgraphs which cover cliques of $G$ (i.e. any clique of $G$ is a clique in either $G_{1}$ or else $G_{2}$.) Then, $$\reg (G)\leq\max\{\reg(G_{1}),\reg (G_{2}),\reg (G_{1}\cap G_{2})+1\}.$$
\end{thm}

\begin{proof}
Let $W$ be an induced subgraph of $G$. Let $W_1=W\cap G_1$ and $W_2=W\cap G_2$.
We claim that a subcomplex of $\Delta W$ 
is the union of subcomplexes of $\Delta W_1$ and $\Delta W_2$. 
Let $F=(v_{1},\dots,v_{t})$ be a face in $\Delta W$.
Then $G(F)$ is a clique in $G$, and since $G_1$ and $G_2$ cover cliques of $G$, we see that $F$ is a face of either $W_1$ or $W_2$, and the claim follows.
Additionally, we have $\Delta(W_{1}\cap W_{2})=\Delta W_1 \cap \Delta W_2 $. 

Now, we prove the main inequality. Let $m=\max\{\reg(G_{1}),\reg(G_{2}),\reg(G_{1}\cap G_{2})+1\}$.
Given any induced subgraph $W$, by the Mayer-Vietoris sequence \cite{MR1867354}*{p.149}, we have
following exact sequence of complexes
\[
	\cdots\rightarrow\tilde{H}_{i}(\Delta(W_{1}\cap W_{2}))\rightarrow\tilde{H}_{i}(\Delta W_{1})\oplus \tilde{H}_{i}(\Delta W_{2})\rightarrow\tilde{H}_{i}(\Delta W)\rightarrow\tilde{H}_{i-1}(\Delta(W_{1}\cap W_{2}))\rightarrow\cdots
\]
 Since regularity of $G_{1} \cap G_{2}$ 
 is at most $m-1$, we have $\tilde{H}_{i}(\Delta (W_{1}\cap W_{2}))=0$ for all $i\geq m-2$. 
 Therefore, $\tilde{H}_{i}(\Delta W)\simeq\tilde{H}_{i}(\Delta W_{1}) \oplus \tilde{H}_{i}(\Delta W_{2})$ for all $i\geq m-1$. Since both $G_{1}$ and $G_{2}$ have regularity at most $m$, $\tilde{H}_{i}(\Delta W_{1})=\tilde{H}_{i}(\Delta W_{2})=0$ for all $i\geq m-1$. Thus, $\tilde{H}_{i}(\Delta W)=0$ for all $i\geq m-1$ and regularity of $G$ is at most $m$.
\end{proof}

Our first application deals with the case of defining $G_1$ and $G_2$ via a cutset.
\begin{thm}[Cut-set/Separator decomposition]\label{thm:separator}
 Let $T$ be an induced subraph of $G$ such that the induced graph $G\setminus T$ is disconnected. Let $C_{1},\dots,C_{k}$ be the connected components of $G\setminus T$ and $G_{i}$ be induced subgraphs on vertices of $C_{i}$ and $T$ for $i=1,\dots,k$. Then, $\reg(G)\leq\max\{\reg(G_{i})_{i=1,\dots,k},\reg(T)+1\}$.
\end{thm}

\begin{proof}
Let $G_1$ be the induced subgraph on vertices of $C_1$ and $T$ and let $G_{1}'$ be the induced subgraph on ${\displaystyle\cup_{i=2}^{k}V(C_{i})}\cup V(T)$. In other words, $G_{1}'=G\setminus C_{1}$. Then, we can see that $G_{1}$ and $G_{1}'$ cover all cliques of $G$. Indeed, if a vertex in $C_{1}$ and a vertex in ${\displaystyle\cup_{i=2}^{k}}C_{i}$ are contained in a clique in $G$, the induced subgraph on the two vertices must be an edge of $G$. However, it is not possible because $C_{1}$ and ${\displaystyle \cup_{i=2}^{k}}C_{i}$ are disjoint. Therefore, two induced subgraphs $G_{1}$ and $G_{1}'$ cover all cliques in $G$. Then, by Theorem \ref{thm:maindecomp}, we have
\[
\reg(G)\leq\max\{\reg(G_{1}),\reg(G_{1}'),\reg(T)+1\}.
\]
Now, let $G_{j}'$ be the induced subgraph on vertices of $C_{j+1},\dots,C_k$, and $T$ for $j=2,\dots,k$.
Then, by the same process,
\[
\reg(G_{j-1}')\leq\max\{\reg(G_{j}),\reg(G_{j}'),\reg(T)+1\}
\]
Thus,
\begin{align*}
\reg(G) & \leq \max\{\reg(G_{1}),\reg(G_{1}'),\reg(T)+1\}\\
 & \leq \max\{\reg(G_{1}),\reg(G_{2}),\reg(G_{2}'),\reg(T)+1\}\\
  &\vdots\\
 & \leq \max\{\reg(G_{i})_{i=1,\dots,k},\reg(T)+1\}.
\end{align*}

\end{proof}

We call $T$ in Theorem \ref{thm:separator} a separator of $G$.
Given any graph $G$, an open neighborhood of a vertex of $G$ is a separator of $G$, and so we obtain the following Vertex Neighborhood Decomposition.

\begin{cor}[Vertex Neighborhood Decomposition] \label{vertexdecomp}
Let $v$ be any vertex of a graph $G$.
Then, \[\reg(G)\leq\max\{\reg(G \setminus v),\reg(N_{G}(v))+1\}.\]
\end{cor}
\begin{proof}
By Theorem \ref{thm:separator}, we have $\reg(G)\leq\max\{\reg(G \setminus v),\reg(N_{G}[v]),\reg(N_{G}(v))+1\}$, where $N_{G}(v)$ is the open neighborhood of $v$ in $G$ and $N_{G}[v]$ is the closed neighborhood of $v$ in $G$ (see Section \ref{sec:2} for definitions). So, it suffices to show that $\reg(N_{G}[v])=\reg(N_{G}(v))$. This follows by a simple application of Hochster's formula, since the clique complex $\Delta H$ of an induced subgraph $H$ of $N_{G}[v]$ with $v\in H$ is contractible.
\end{proof}

So far, we have only considered graph decompositions coming from induced subgraphs, but we now define a useful decomposition where this is not the case. 
Let $M$ be a subgraph of $G$. Let $G_{M}$ be the induced subgraph of $G$ on vertices in $M$ and vertices of $G$ which are adjacent to both vertices of some edge of $M$. 
Namely, $G_{M}=G[V(M)\cup W]$ where $W$ is a subset of vertices in $G$ such that $k\in W$ if $ik\in E(G)$ and $jk\in E(G)$ for some $ij\in E(M)$. Then, we have following decomposition theorem.

\begin{thm}\label{subgraphdecomp}
Let $M$ be a subgraph of a graph $G$.
Then, $\reg(G)\leq\max\{\reg(G-M), \reg (G_{M}),\reg(G_{M}-M)+1\}$ 
\end{thm}
\begin{proof}
We first claim that $G-M$ and $G_{M}$ cover all cliques of $G$. Let $F$ be any clique in $G$. If $F$ does not contains any edges in $M$, then $G-M$ contains the clique $F$. Suppose that $F$ contains some edges of $M$. If all vertices in $F$ is contained in $M$, then $F$ is contained in $G_{M}$ since $G_M$ contains $M$. If $v$ is any vertex in $F$ outside of $M$, then $uv,wv\in F$ for some $uw\in E(F\cap M)$. This implies that $v\in V(G_{M})$ and so $F\subseteq G_{M}$ since both $F$ and $G_{M}$ are induced subgraphs of $G$. Additionally, the intersection of $G-M$ and $G_{M}$ is $G_{M}-M$. Indeed, $V(G_{M}\cap (G-M))=V(G_{M}\cap G)=V(G_{M})$ and $E(G_{M}\cap(G-M))=E(G_{M}-M)$. Thus, by Theorem \ref{thm:maindecomp},  $\reg(G)\leq\max\{\reg(G-M), \reg (G_{M}),\reg(G_{M}-M)+1\}$
\end{proof}

 Similarly to vertex-neighborhood decomposition in Corollary \ref{vertexdecomp}, if we take $M$ to be an edge $e=ij$ in Theorem \ref{subgraphdecomp}, then we can bound regularity of $G$ by regularity of two subgraphs.
 
 \begin{cor}[Edge-neighborhood decomposition]\label{edgedecomp}
Let $G$ be a graph and $e=ij$ be an edge in $G$. 
 Then, $\reg(G)\leq\max\{\reg(G-e),\reg(G_{e}-e)+1\}$.
\end{cor}
\begin{proof}

By Theorem \ref{subgraphdecomp}, it suffices to show that $\reg(G_{e})\leq\reg(G_{e}-e)+1$ for edge $e$. Indeed, for any graph $G$, $\reg(G)\leq\reg(G \setminus v)+1$ for any vertex $v$ by Corollary \ref{vertexdecomp} , and so we have 
\begin{align*}
\reg(G_{e}) & \leq  \reg(G_{e} \setminus i)+1 \\
 & \leq  \reg(G_{e} - e)+1,
\end{align*}
for the edge $e=ij$ because $G_{e}\setminus i$ is an induced subgraph of $G_{e}-e$.
\end{proof}
\noindent We will use this decomposition to describe complements of bipartite graphs that have regularity $3$ in Section \ref{sec:bipart}.

\section{Hereditary Families}\label{sec:hereditary}
Let $\mathscr{G}$ be a family of graphs. We call $\mathscr{G}$ a hereditary family if it is closed under taking induced subgraphs, or equivalently under deleting vertices.

\begin{thm}[Hereditary theorem] \label{thm:hereditary2} Let $\mathscr{G}$ be a hereditary family with the following property:
there exists $t\in \mathbb{N}$, such that for any $G \in \mathscr{G}$ there is a separator $G'$ of $G$ with $\reg (G') \leq t$. Then regularity of any $G \in \mathscr{G}$ is at most $t+1$.
\end{thm}
\begin{proof}
Let $\mathscr{G}$ be a hereditary family with the above property for some $t\in \mathbb{N}$. We will induct on the number of vertices $n$ in graphs of $\mathscr{G}$. The base case $n=1$ is trivial, since $t\geq 0$ and $\mathscr{G}$ includes the one vertex graph. Now consider the inductive step. Let $G\in \mathscr{G}$ be a graph on $n+1$ vertices, and let $G'$ be a separator of $G$. Applying Theorem \ref{thm:separator} with $T=G'$, we get the desired inequality by the induction assumption.
\end{proof}

Chordal graphs form a hereditary family, and it is known in \cite{MR0130190} that any chordal graph contains a vertex $v$ such that neighborhood of $v$ is a complete graph. Therefore we immediately obtain the following result of Fr\"{o}berg:

\begin{cor}\label{cor:chordal}
Let $G$ be a chordal graph. Then regularity of $G$ is at most $2$.
\end{cor}
Moreover, we can see that regularity of any hole is at least $3$ and therefore chordal graphs are the only graphs of regularity at most $2$.
On the other hand, by combining Fr\"{o}berg's result with neighborhood decomposition \ref{vertexdecomp} we can give a criterion for graphs that have regularity at most $3$:
\begin{cor} \label{cor:hereditary}
Let $\mathscr{G}$ be a hereditary family of graphs with the following property: for any $G \in \mathscr{G}$ there is a vertex $v$ of $G$ which has a chordal neighborhood. Then regularity of any $G \in \mathscr{G}$ is at most $3$.
\end{cor}
To illustrate the power of the above Corollary \ref{cor:hereditary}, we give a quick proof of a generalization of a result by Nevo \cite{MR2739498}*{Section 5}. Let $F'$ be a graph on four vertices consisting of an isolated vertex and a triangle. He showed that if $G$ does not contain $F'$ and a four-cycle as induced subgraphs then regularity of $G$ is at most three. We note that not containing a four-cycle as an induced subgraph corresponds to $G$ satisfying condition $N_{2,2}$.
Let $F$ be a graph on five vertices consisting of an isolated vertex and two triangles sharing an edge. We show that if $G$ does not contain a four-cycle and $F$ as induced subgraphs, then regularity of $G$ is at most $3$, which is a weaker condition on $G$. 
\begin{cor}\label{2-fan}
Let $\mathscr{G}$ be the hereditary family of graphs that do not contain $F$ and the four cycle as induced subgraphs. Then regularity of any $G \in \mathscr{G}$ is at most $3$. 
\end{cor}
\begin{proof} We will show that any $G \in \mathscr{G}$ contains a vertex with a chordal neighborhood. Suppose not, and let $G \in \mathscr{G}$ be a graph such that no vertex of $G$ has a chordal neighborhood. Let $v$ be the vertex of minimal degree in $G$. Observe that $v$ is not connected to all vertices of $G$, otherwise $G$ is the complete graph, which is a contradiction. It follows by our assumption that $N_G(v)$ contains a hole $C$ of length at least $5$, and there exists $w\in G$ such that $v$ is not connected to $w$. Since $G$ is $F$-free we see that $w$ must be connected to two non-adjacent vertices $u_1, u_2$ of $C$. But then the induced subgraph on $u_1,v,u_2, w$ is a $4$-cycle, which is a contradiction.
\end{proof}

We also generalize Corollary \ref{2-fan} to the case where $G$ does not contain larger cycles as induced subgraphs. Recall that a graph $G$ not containing an $\ell$-hole for $\ell=4,...,p+2$ with $p\geq 2$ is equivalent to $G$ satisfying condition $N_{2,p}$. Let a fan $F_{i}$  for $i\geq 1$ be the graph consisting of an isolated vertex and the graph join of a path on $i+1$ vertices and a distinct vertex. With essentially the same proof as Corollary \ref{2-fan} we can also show the following:
\begin{cor}\label{l-fan}
If for some $i\geq 2$ a graph $G$ is $\ell$-hole free for $\ell=4,...,i+2$ and does not contain $F_{i}$ as an induced subgraph, then regularity of $G$ is at most $3$.

\end{cor}
It is known that if $G$ is perfect and does not contain $4$-holes or if $G$ is even-hole free, then there is a vertex in $G$ whose neighborhood is chordal (for $4$-free perfect graphs see \cite{MR1852504} and for even-hole free graphs see \cite{MR2292535}). Moreover, both $4$-hole free perfect graphs and even-hole free graphs form hereditary families. Thus, we obtain another criterion to make graphs to have regularity $3$.
\begin{cor}\label{even,odd holefree}
If $G$ is perfect and does not contain $4$-holes, or if $G$ is even-hole free then regularity of $G$ is at most $3$.
\end{cor}

It follows from the Strong Perfect Graph Theorem \cite{MR2233847}, that $G$ is perfect and $4$-hole free if and only if $G$ is $4$-hole free and also odd-hole free. Thus Corollary \ref{even,odd holefree} implies that if $G$ is $4$-hole free, and regularity of $G$ is at least $4$, then $G$ must contain both even and odd holes. This observation is used for improving a bound on regularity in Section \ref{sec:hole-free}.

\section{Complements of bipartite graphs} \label{sec:bipart}
Fern\'{a}ndez-Ramos and Gimenez gave an explicit description of bipartite graphs associated to edge ideals that have regularity $3$ in \cite{MR3199032}. 
We give an independent proof of their result by using Edge Neighborhood Decomposition. Since we consider non-edge ideals, we work with complements of bipartite graphs.

Let $G$ be the complement of a bipartite graph $H$ with bipartition of vertices $X$ and $Y$. Let $B$ be the subgraph of $G$ with $V(B)=V(G)$ and the edge set consisting of edges of $G$ between vertices in $X$ and vertices in $Y$. We call $B$ the bipartite part of $G$. We recall chordal bipartite graphs \cite{MR2063679}*{Section 12.4}.

\begin{defn}
A chordal bipartite graph is a bipartite graph which contains no induced cycles of length greater than four.
\end{defn}

It is shown in \cite{MR493395} that any chordal bipartite graph $G$ with bipartition of vertices $X$ and $Y$ contains an edge $ij$ for $i\in X$ and $j\in Y$ such that the induced subgraph on vertices of $N_{G}(i)$ and $N_{G}(j)$ is a complete bipartite graph.
Such an edge $ij$ is called a \textit{bisimplicial} edge. 
Additionally, it is known in \cite{MR493395} that the subgraph $G-ij$ is again a chordal bipartite graph.  This implies that subgraphs obtained by deleting a bisimplicial edge from a chordal bipartite graph are also chordal bipartite graphs. 

Combining Corollary \ref{edgedecomp} with property of chordal bipartite graph, we get an exact description of complements of bipartite graphs of regularity $3$.

\begin{thm} \label{thm:bipartite} 
Let $G$ be the complement of a bipartite graph. Regularity of $G$ is $3$ if and only if $G$ contains a hole and the bipartite part $B$ of $G$ is chordal bipartite.
\end{thm}
\begin{proof}
Suppose that the complement $G$ of a bipartite graph $H$ has at least one hole and the bipartite part $B$ of $G$ is a chordal bipartite graph. Since $G$ contains at least one hole, regularity of $G$ is at least $3$. To show that regularity of $G$ is at most $3$ we induct on the number of edges $\ell$ in $B$. The base case $\ell=0$ is simple, since $G$ is then chordal and therefore $\reg(G) \leq 2$. Now we consider the induction step.
Let $G$ be the complement of a bipartite graph such that its bipartite part $B$ is a chordal bipartite graph with $\ell+1$ edges.
Then $B$ contains a bisimplicial edge $e$. By Theorem \ref{edgedecomp}, 
\[
\reg(G)\leq\max\{\reg(G-e),\reg(G_{e}-e)+1\}.
\]
Since $e$ is a bisimplicial edge in $B$, $G_{e}-e$ is a chordal graph, and $\reg(G_{e}-e)\leq 2$. Additionally, $\reg(G-e)\leq 3$ by the induction assumption, and the desired result follows.

Conversely, suppose that bipartite part $B$ of $G$ contains a hole of length at least $6$. We claim that $\Delta G$ contains a subcomplex whose $2$nd (reduced) homology is not zero. Let $G'$ be the subgraph of $G$ induced by vertices that form the shortest hole in $B$. Let $X'$ and $Y'$ be the partitions of vertices $G'$ (induced from the partition of vertices of $G$). 
Let $v$ be any vertex of $X'$. Then, the closed neighborhood $N_{G'}[v]$ and the deletion $G' \setminus v$ of $v$ cover cliques of $G'$. Observe that $\widetilde{H}_{1}(\Delta N_{G'}[v])=\widetilde{H}_{1}(\Delta (G'\setminus v))=0$ since $\Delta N_{G'}[v]$ is contractible, and any hole in $G'-v$ is covered by cliques of size $3$, but $\widetilde{H}_{1}(\Delta N_{G'}(v))\not=0$ since $N_{G'}(v)$ contains a hole (of length $4$). Since $\widetilde{H}_{2}(\Delta G')\rightarrow\widetilde{H}_{1}(\Delta N_{G'}(v))$ is surjective by the Mayer-Vietoris sequence, $\widetilde{H}_{2}(\Delta G')\not= 0$ , and this implies that regularity of $G$ is at least $4$.
\end{proof}

\section{Regularity and Genus}
\label{sec:genus}
The following bound on regularity is well-known in \cite{MR1417301}*{Lemma 2.1} (or see \cite{MR3249840} for a geometric proof), but we provide a short proof for the sake of completeness.
\begin{lem}\label{lem:genus}
If the number of vertices of $G$ is at most $2n-1$, then regularity of $G$ is at most $n$.
\end{lem}

\begin{proof}
We use induction on $n$. For $n=1$, regularity is obviously at most $1$ since there are no generators in the non-edge ideal of the graph. Assume that any graph with at most $2\ell-1$ vertices has regularity at most $\ell$. Let $G$ be a graph on $2\ell+1$ vertices. Note that by Corollary \ref{vertexdecomp} we can delete a vertex $v$ without changing regularity if $\reg(G)>\reg(N_G(v))+1$. After deleting such vertices, if possible, let $v$ be the vertex of minimal degree in $G$. If the degree of $v$ is $2\ell$, then $G$ is a complete graph (which has regularity $1$). Therefore, we can assume that degree of $v$ is at most $2\ell-1$. Then, we have  \[ \reg(G)\leq \reg(N_G(v))+1\leq \ell+1, \]
since $N_G(v)$ contains at most $2\ell-1$ vertices.
\end{proof}
In fact, the bound in Lemma \ref{lem:genus} is tight. Let $K_{n(2)}$ be the complete $n$-partite graph, with each part of size two. Since the ideal of $K_{n(2)}$ is a complete intersection of $n$ quadrics, its minimal resolution is given by the Koszul complex. Thus regularity of $K_{n(2)}$ is $n+1$.
We also note that $K_{n(2)}$ is a perfect graph on $2n$ vertices.

Recall that the \textit{genus} of a graph $G$ is the minimal genus of an orientable surface $S_g$ into which $G$ can be embedded (See \cite{MR1852593} for reference). Note that any graphs can be embedded into an orientable surface $S_g$ for some genus $g$ and the genus of graphs inscribes a topological complexity of graphs. By using the Lemma \ref{lem:genus}, we can immediately give an alternative proof of a result in \cite{MR3249840} that any planar graphs have regularity at most $4$ and it is tight. We note that this is the case of genus $0$ and we can provide bounds on regulairty of graphs in terms of arbitrarily genus.

\begin{thm}\label{thm:genus}
    Let $g$ be the genus of a graph $G$. Then, regularity of $G$ is at most $\lfloor1+\sqrt{1+3g}\rfloor+2$.
\end{thm}

\begin{proof}
	Let $|V|$ be the number of vertices, $|E|$ be the number of edges, and $|F|$ be the number of (2-dimensional) faces in the embedding of $G$. By considering the Euler characteristic of the surface $S$ into which $G$ is embedded, we see that $|V|-|E|+|F|=2-2g$. Recall that ${\displaystyle 2|E|=\sum_{v\in V}\deg(v)=\sum_{F\in\Delta_{2}}\ell_{F}}$ where $\Delta_{2}$ is the set of 2-cells in the embedding and $\ell_{F}$ is the number of edges in the face $F$. In particular, ${\displaystyle 2|E|=\sum_{F\in\Delta_{2}}\ell_{F}\geq3|F|}$ since $\ell_{F}\geq3$ for any face $F$. Let $d$ be the minimal degree of $G$. Then, ${\displaystyle 2|E|=\sum_{v\in V}\deg(v)}\geq d|V|$.
	Therefore, 
	\begin{align*}
	2-2g & =  |V|-|E|+|F| \\
	& \leq  |V|-|E|+\frac{2}{3}|E| = |V|-\frac{1}{3}|E|\\
	& \leq  |V|-\frac{d}{6}|V| = \frac{6-d}{6}|V|.
	\end{align*}
	Moreover, we can see that $|V|\geq d+2$ since $d\leq\deg(v)\leq|V|-2$. (Note that, if $d=|V|-1$, the graph is complete graph, which can be excluded) Thus, 
	\[ 6(2g-2)\geq(d-6)|V|\geq(d-6)(d+2)\Rightarrow0\geq d^{2}-4d-12g. \]
	This implies that $d\leq2+\sqrt{4+12g}=2+2\sqrt{1+3g}$. Let $v$ be the vertex of degree $d$. Then, $\reg(N_G(v))\leq\lfloor\frac{1}{2}\lfloor2+2\sqrt{1+3g}\rfloor\rfloor+1=\lfloor1+\sqrt{1+3g}\rfloor+1$. By Theorem \ref{thm:hereditary2}, $\reg(G)\leq\lfloor1+\sqrt{1+3g}\rfloor+2$.
\end{proof}

Note that this bound is indeed tight. It is known in \cite{MR0349461}*{Section 4.4} that the genus of $2$-regular complete $n$-bipartite graphs $K_{n(2)}(=K_{2,2,\dots,2})$ is at least  $\frac{(n-3)(n-1)}{3}$. Moreover, the genus of $K_{n(2)}$ is exactly $\frac{(n-3)(n-1)}{3}$   if $n\not\equiv 2\mod 3$ by \cite{MR0485485}. In this case, we have $\reg(K_{n(2)})=n+1$ and the right hand side of inequality in Theorem \ref{thm:genus} is $\lfloor1+\sqrt{1+3\frac{(n-3)(n-1)}{3}}\rfloor+2=n+1$.

\section{Bounds on regularity of graphs without small holes} \label{sec:hole-free}
Even though regularity of a graph can depend linearly on the number of vertices $n$, if $G$ does not contain small holes, then regularity of $G$ can be bounded from above by a logarithmic function of $n$. It was shown in \cite{MR2188445} that absence of small holes corresponds to the ideal satisfying property $N_{2,p}$ for some $p\geq 2$.

\begin{thm}
	Let $p\geq 2$ and $I(G)$ be the non-edge ideal corresponding to a graph $G$. Then, the followings are equivalent.
    \begin{enumerate}
	\item The minimal graded free resolution of $I(G)$ is $(p-1)$-step linear.
	\item The graph $G$ does not contain a hole $C_{i}$ of length $i$ for $i\leq p+2$.
	\item $I(G)$ satiesfies $N_{2,i}$ for all $2\leq i\leq p$.
 	\end{enumerate}

\end{thm}
It was shown in \cite{MR3070118} that if $G$ satisfies $N_{2,p}$ for $p \geq 2$, then $$\reg(G) \leq \log_{\frac{p+3}{2}}\frac{n-1}{p}+3. $$ We also provide (a similar and) asymptotically better upper bound on regularity of graphs.
\begin{thm}\label{thm:holefree}
	Suppose that $G$ satisfies property $N_{2,p}$ for $p\geq 2$. Then, 
	\[ 
    \reg(G)\leq \min \{\log_{\frac{p+3}{2}}\big(\frac{n(p+1)}{p(p+3)}\big)+3,\log_{\frac{p+4}{2}}\big(\frac{n(p+2)}{(p+1)(p+4)}\big)+4 \}.
    \]
\end{thm}

\begin{proof}
	Given a graph $G$, there is an induced subgraph $G_{0}$ such that $\reg(G)=\reg(G_{0})=\reg(N_{G_{0}}(v))+1$ for any vertex $v$ in $G_{0}$. 
 Indeed, we can keep deleting vertices $y$ such that $\reg(G)= \reg(G\setminus y)$ until we arrive at a graph $G_0$, where $\reg(G_0\setminus v)=\reg(G_0)-1$ for any vertex $v$ of $G$.	
    Then, by Corollary \ref{vertexdecomp} we have $\reg(G_{0})=\reg(N_{G_{0}}(v))+1$ for any vertex $v$ in $G_{0}$.
 We call such $G_0$ a trimming of $G$. Note that a trimming is not unique.

 Let $x_{0}$ be a vertex of minimal degree in $G_{0}$. Let $G_1$ be a trimming of the open neighborhood $N_{G_0}(x_0)$ of $x_0$ in $G_0$. Now we repeat this process: let $x_i$ be a vertex of minimal degree in $G_i$ and let $G_{i+1}$ be a trimming of the open neighborhood of $x_i$ in $G_i$.
 We obtain a sequence of induced subgraphs $G_{i}$ of $G$ such that
	\[ \reg(G)=\reg(G_{0})=\reg(G_{1})+1=\cdots=\reg(G_{t})+t .\]
 Let $\ell$ be the maximal integer such that $G_{\ell}$ contains a hole, and let $C_m$ be the hole in $G_{\ell}$ of smallest length $m$, with $m\geq p+3\geq 5$. Note that $C_m$ is a hole that is present in all graphs $G_i$, with $0\leq i\leq \ell$.
 We use $d_{i}$ to denote the degree of $x_{i}$ in $G_{i}$.

  We claim that for $1\leq i \leq \ell$ the sum of the degrees of vertices of $C_m$ in $N_{G_{\ell-i}}[x_{\ell-i}]$ is at most \[md_{\ell-i}-\frac{m^i(m-3)}{2^{i-1}},\] which we prove by induction on $i$. The base case is $i=1$: a vertex of $C_m$ is connected to exactly two vertices of $C_m$ and can be connected to all other vertices in $N_{G_{\ell-1}}[x_{\ell-1}]$.
 Therefore, the sum of degrees of vertices of $C_m$ is at most  $2m+m(d_{\ell-1}+1-m)=md_{\ell-1}-m(m-3)$. 
 
 For the inductive step, assume that the sum of the degrees of vertices of $C_m$ in $N_{G_{\ell-i+1}}[x_{\ell-i+1}]$ is at most $md_{\ell-i+1}-\frac{m^{i-1}(m-3)}{2^{i-2}}$. Observe that any vertex in $G_{\ell-i+1}$ not connected to $x_{\ell-i+1}$ can be adjacent to at most two vertices of $C_m$. Otherwise $G_{\ell-i+1}$ is forced to have a $4$-hole, which is a contradiction.
 Since degree of $x_{\ell-i+1}$ in $G_{\ell-i+1}$ is at least the degree of any vertex of $C_m$ is $G_{\ell-i+1}$ we see that
  there are at least
   \begin{equation}\label{eqn:extravert}\frac{1}{2} \left(md_{\ell-i+1}-(md_{\ell-i+1}-\frac{m^{i-1}(m-3)}{2^{i-2}})\right)=\frac{m^{i-1}(m-3)}{2^{i-1}}\end{equation} vertices in $G_{\ell-i+1}\setminus N_{G_{\ell-i+1}}[x_{\ell-i+1}]$.
   
    Any vertex of $N_{G_{\ell-i}}[x_{\ell-i}]$ belongs to exactly one of  $N_{G_{\ell-i}}[x_{\ell-i}]\setminus G_{l-i+1}$, or $G_{l-i+1}\setminus N_{G_{l-i+1}}[x_{l-i+1}]$, or $N_{G_{l-i+1}}[x_{l-i+1}]$. 
  As before, any vertex of  $G_{l-i+1} \setminus N_{G_{\ell-i+1}}[x_{\ell-i+1}] $ can be adjacent to at most two vertices in $C_m$, and a vertex of $C_m$ can be adjacent to all vertices of $N_{G_{\ell-i}}[x_{\ell-i}]\setminus G_{\ell-i+1}$. Therefore, 
    \begin{align*}
		\sum_{v\in C_{m}}\deg_{N_{G_{\ell-i}}[x_{\ell-i}]}(v) & \leq & m|N_{G_{\ell-i}}[x_{\ell-i}]\setminus G_{\ell-i+1}|+2|G_{\ell-i+1} \setminus N_{G_{\ell-i+1}}[x_{\ell-i+1}]|+\sum_{v\in C_{m}}\deg_{N_{G_{\ell-i+1}}[x_{\ell-i+1}]}(v). \\
	\end{align*}
	Using the induction assumption on $\sum_{v\in C_{m}}\deg_{N_{G_{\ell-i+1}}[x_{\ell-i+1}]}(v)$ we see that
	\[\sum_{v\in C_{m}}\deg_{N_{G_{\ell-i}}[x_{\ell-i}]}(v)  \leq md_{\ell-i}-(m-2)\big(|G_{\ell-i+1}\setminus N_{G_{\ell-i+1}}[x_{\ell-i+1}]|\big)-\frac{m^{i-1}(m-3)}{2^{i-2}}. \]
By \eqref{eqn:extravert} we see that $$(m-2)\big(|G_{\ell-i+1}\setminus N_{G_{\ell-i+1}}[x_{\ell-i+1}]|\big)+\frac{m^{i-1}(m-3)}{2^{i-2}}\geq \frac{m^{i}(m-3)}{2^{i-1}}, $$
and therefore
\[\sum_{v\in C_{m}}\deg_{N_{G_{\ell-i}}[x_{\ell-i}]}(v)  \leq md_{\ell-i}-\frac{m^{i}(m-3)}{2^{i-1}},\] as desired.
	The argument above shows that there are at least $\frac{m^{i}(m-3)}{2^{i}}$ vertices in $G_{\ell-i}\setminus N_{G_{\ell-i}}[x_{\ell-i}]$. Since $G_{\ell-i+1}$ is a subgraph of $N_{G_{\ell-i}}[x_{\ell-i}]$, we see that $$|G_{\ell-i}|-|G_{\ell-i+1}|\geq \frac{m^{i}(m-3)}{2^{i}}.$$ 
Therefore, $$|G_{\ell-i}|\geq \sum_{t=1}^{i}\frac{m^{t}(m-3)}{2^{t}}+m,$$
and by summing the above geometric series we see that
\[ |G_{\ell-i}|\geq\frac{m^{i+1}(m-3)}{2^{i}(m-2)}. \] 
Plugging in $i=\ell$, we see that
\[ n\geq|G_{0}|\geq\frac{m^{\ell+1}(m-3)}{2^{\ell}(m-2)}\geq\frac{p(p+3)^{\ell+1}}{2^{\ell}(p+1)}.\]
 Thus, \[\reg(G)\leq\reg(G_{\ell+1})+\ell+1\leq\log_{\frac{p+3}{2}}\left(\frac{n(p+1)}{p(p+3)}\right)+3,\]
    which gives us the first upper bound.
    
For the second upper bound, we observe that if regularity of $G$ is at least four, then $G$ contains both even and odd holes by Corollary \ref{even,odd holefree}. 
With the same setting above, regularity of $N_{G_{\ell-2}}(x_{\ell-2})$ (or equivalently, $G_{\ell-1}$) is four. Let $m$ be the length of the smallest hole in $N_{G_{\ell-2}}(x_{\ell-2})$. Then $N_{G_{\ell-2}}(x_{\ell-2})$ must also contain a hole of size  $m+2\alpha+1$ for some positive integer $\alpha$. We can now apply the same process as above to bound number of vertices of $G_{\ell-i}$ using this larger hole to obtain
\[ |G_{\ell-i}| \geq \frac{(m+2\alpha+1)^{i}(m+2\alpha-2)}{2^{i-1}(m+2\alpha-1)}, \]
for $2 \leq i \leq \ell$.
By taking $i=\ell$, we see that
\[ n\geq|G_{0}| \geq \frac{(m+2\alpha+1)^{\ell}(m+2\alpha-2)}{2^{\ell-1}(m+2\alpha-1)}\geq\frac{(p+4)^{\ell}(p+1)}{2^{\ell-1}(p+2)}. \]  
Thus, \[\reg(G)\leq\reg(G_{\ell+1})+\ell+1\leq\log_{\frac{p+4}{2}}\frac{n(p+2)}{(p+1)(p+4)}+4.\]
\end{proof}
Note that the former term in the bound in Theorem \ref{thm:holefree} is slightly better (if $n\geq \frac{p+3}{2}$) than the bound in \cite{MR3070118}*{Theorem 4.9} and the former term will be smaller than the latter term if the size of a graph is relatively small. However, the latter term of the bound is better asymptotically. \\

\noindent \textbf{Acknowledgements.} The authors would like to thank Hailong Dao for useful conversations. Grigoriy Blekherman was partially supported by NSF Grants DMS-1352073 and DMS-1901950. Jaewoo Jung was partially supported by NSF Grant DMS-1901950. 

\bibliographystyle{plain}

\end{document}